\documentclass[a4paper, 12pt, reqno]{amsart}
\usepackage{amsmath,amsfonts,amssymb,graphics,mathrsfs,amsthm,eurosym,verbatim,enumerate}
\usepackage[marginratio=1:1,tmargin=117pt, height=650pt]{geometry}
\usepackage[usenames, dvipsnames]{xcolor}

\usepackage[normalem]{ulem}
\usepackage{bm}
\usepackage{tikz-cd}
\usepackage[colorlinks=true,linkcolor=blue!60!black,citecolor=teal!80!black,urlcolor=RoyalBlue]{hyperref}

\numberwithin{equation}{section}
\makeatletter
\def\blfootnote{\xdef\@thefnmark{}\@footnotetext}
\makeatother
\theoremstyle{plain}
\newtheorem{theorem}{Theorem}[section]
\newtheorem{proposition}[theorem]{Proposition}

\newtheorem{lemma}[theorem]{Lemma}
\newtheorem{definition}[theorem]{Definition}

\newcommand*{\defeq}{\mathrel{\vcenter{\baselineskip0.5ex \lineskiplimit0pt
 \hbox{\scriptsize.}\hbox{\scriptsize.}}}%
 =}
\newtheorem*{remark}{Remark}%[section]
\newtheorem{observation}[theorem]{Observation}

\theoremstyle{remark}
\newtheorem*{claim}{\normalfont CLAIM}
\newcommand{\C}{{\mathbb{C}}}

\newcommand{\R}{{\mathbb{R}}}
\newcommand{\Z}{{\mathbb{Z}}}
\newcommand{\N}{{\mathbb{N}}}

\newcommand{\M}{{\mathcal{M}}}
\newcommand{\Mlog}{{\mathcal{M}_{\log}}}
\newcommand{\B}{{\mathcal{B}}}
\newcommand{\tef}{transcendental entire function}
\newcommand{\qfor}{\quad\text{for }}
\newcommand{\uldelta}{\underline{\delta}}
\newcommand{\Rea}{\operatorname{Re }}

\newenvironment{subproof}{\begin{proof}[Proof of claim]}{%
               \end{proof}}
\begin{document}
\title[Variations on a theme of Hardy]{Variations on a theme of Hardy concerning the maximum modulus}
\author[{L. Pardo-Sim\'on \and D. J. Sixsmith}]{Leticia Pardo-Sim\'on \and David J. Sixsmith}
\address{Institute of Mathematics of the Polish Academy of Sciences\\ ul. \'Sniadeckich~8\\
00-656 Warsaw\\ Poland\textsc{\newline \indent \href{https://orcid.org/0000-0003-4039-5556}{\includegraphics[width=1em,height=1em]{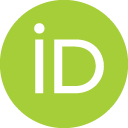} {\normalfont https://orcid.org/0000-0003-4039-5556}}}
}
\email{l.pardo-simon@impan.pl}
\address{School of Mathematics and Statistics\\ The Open University\\
Milton Keynes MK7 6AA\\ UK\textsc{\newline \indent \href{https://orcid.org/0000-0002-3543-6969}{\includegraphics[width=1em,height=1em]{orcid2.png} {\normalfont https://orcid.org/0000-0002-3543-6969}}}}
\email{david.sixsmith@open.ac.uk}
\thanks{\vspace{3pt}\vspace{3pt}\\ }
\begin{abstract}
In 1909, Hardy gave an example of a transcendental entire function, $f$, with the property that the set of points where $f$ achieves its maximum modulus, $\mathcal{M}(f)$, has infinitely many discontinuities. This is one of only two known examples of such an entire function.

In this paper we significantly generalise these examples. In particular, we show that, given an increasing sequence of positive real numbers, tending to infinity, there is a transcendental entire function, $f$, such that $\mathcal{M}(f)$ has discontinuities with moduli at all these values. We also show that the transcendental entire function lies in the much studied Eremenko-Lyubich class. Finally, we show that, with an additional hypothesis on the sequence, we can ensure that $f$ has finite order.
\end{abstract}
\maketitle
\section{Introduction}
Suppose that $f$ is an entire function. We define the \emph{maximum modulus} by
\[
M(r, f) \defeq \max_{|z| = r} |f(z)|, \qfor r \geq 0.
\]
Following \cite{Sixsmithmax}, we denote the set of points where $f$ achieves its maximum modulus, which we call the \emph{maximum modulus set}, by $\M(f)$. In other words
\begin{equation}
\label{Mdef}
\M(f) \defeq \{ z \in \C \colon |f(z)| = M(|z|,f) \}.
\end{equation} 

If $f(z) \defeq cz^n$, for $c \in \C$ and $n \geq 0$, then $\M(f) = \C$; we are not interested in this case. Otherwise, see, for example, \cite[Theorem 10]{valironlectures} or \cite{Blumenthal}, it is known that $\M(f)$ consists of an, at most countable, union of closed \emph{maximum curves}, which are analytic except at their endpoints, and may or may not be unbounded.

In this paper we are interested in discontinuities in the maximum modulus set, which we define as follows. 
\begin{definition}\normalfont 
Let $f$ be an entire function. We say that $\M(f)$ has a \emph{discontinuity} $z$ \emph{at} $r$, for $r>0$, if there exists a connected component $\Gamma$ of $\M(f)$ such that $z \in \Gamma$ with $|z| = r$, and such that $\Gamma$ contains no points of modulus less than $r$.
% $z$ is the endpoint of a maximal curve of $\M(f)$ but meets no other such curve.
\end{definition}

%Such discontinuities were first defined by Blumenthal \cite{Blumenthal}, who called them \emph{exceptional values of the second kind}. (Blumenthal used the term \emph{exceptional value of the first kind} for a point where two maximal curves meet.) However, he did not give any examples of an entire function whose maximum modulus set has such discontinuities, although he conjectured that there is a cubic polynomial with this property.
Such discontinuities were first studied by Blumenthal \cite{Blumenthal}; see also \cite{retrospect}. However, he did not give any examples of an entire function whose maximum modulus set has such discontinuities, although he conjectured that there is a cubic polynomial with this property; such a polynomial was given in \cite{jassimlondon}.

Hardy \cite{hardy1909} was the first to construct maximum modulus sets with discontinuities. Indeed, he constructed a {\tef}, $g$, such that $\M(g)$ has infinitely many discontinuities. Hardy's example is the function
\begin{equation}
\label{eq:gdef}
g(z) \defeq  \alpha \exp\left(e^{z^2} + \sin z\right), \qfor \alpha > 0, 
\end{equation}
where $\alpha$ is assumed to be sufficiently large. Hardy's calculations are complicated but essentially elementary. He shows that, for large values of $r>0$, $\M(g)$ lies on the real line, precisely where the sine function is non-negative. In other words, away from the origin, $\M(g)$ consists of line segments of the form $[2n\pi, (2n+1)\pi]$, for all sufficiently large (positive or negative) integers $n$.

%In Section~\ref{S.basic}, we discuss an alternative approach to proving this result. That is, we construct a transcendental entire function whose maximum modulus set has infinitely many discontinuities. Although this does not add anything essentially new, it illustrates the ideas used in our more complicated construction used to prove our results below.
%
The only other example in the literature of an entire function with infinitely many discontinuities was given in \cite{tyler}, where it was shown that the function
\begin{equation}
\label{eq:fdef}
f(z) \defeq \exp\left(e^{z^2} + 2z \sin^2 z\right),
\end{equation}
has the property that $\M(f)$ has an infinite number of \emph{isolated points};  that is, degenerate maximum curves. %These seem to be the only such examples in the literature.

Our goal in this paper is to give some generalisations of these results, in the following sense. First we show that it is possible to specify the moduli of the  discontinuities, provided they tend to infinity. At the same time we show that the function constructed can be chosen to lie in the \emph{Eremenko-Lyubich} class, denoted by $\B$, which has been much studied in complex analysis and complex dynamics; see, for example, \cite{eremenkoandlyubich, bishop1, bishop, lasseacta, lassedave, rrrs} and the survey \cite{DaveSurvey}. The class $\B$ consists of those {\tef}s for which the set of \emph{singular values}, in other words those values through which some inverse branch cannot be continued, is bounded. In a sense, these functions are the {\tef}s with properties most resembling those of the polynomials. It is easy to see that the functions in \eqref{eq:gdef} and \eqref{eq:fdef} do not lie in the class $\B$.
\begin{theorem}
\label{theo:main}
Suppose that $(r_n)_{n \in \N}$ is an increasing sequence of positive real numbers tending to infinity. Then there is a {\tef} $f \in \B$ such that $\M(f)$ has a discontinuity at $r_n$, for all $n \in \N$.
\end{theorem}

The \emph{order} of a {\tef} $f$ is defined by
\[
\rho(f) \defeq \limsup_{r \rightarrow\infty} \frac{\log \log M(r,f)}{\log r}.
\]
We show that, with an additional hypothesis on the sequence $(r_n)_{n \in \N}$, the function can be chosen to have \emph{finite} order. It is easy to see that the functions in \eqref{eq:gdef} and \eqref{eq:fdef} do not have finite order.
\begin{theorem}
\label{theo:finite order}
Suppose that $(r_n)_{n \in \N}$ is a sequence of positive real numbers, and that there exists $C>1$ such that $r_{n+1} > Cr_n$, for $n \in \N$. Then there is a {\tef} $f \in \B$ with $\rho(f) < \infty$, such that $\M(f)$ has a discontinuity at $r_n$, for all $n \in \N$.
\end{theorem}

\begin{remark}\normalfont
Tyler \cite{tyler} gave an example of a polynomial $P$ for which $\M(P)$ has a discontinuity; indeed, $\M(P)$ has an isolated point. It seems natural to ask, in view of Theorem~\ref{theo:finite order}, if there is a {\tef} $f$ with $\rho(f) = 0$, such that $\M(f)$ has at least one discontinuity. We have not been able to answer this question, as the techniques of this paper all give rise to functions in class $\B$, and such functions always have order at least a half; this follows readily from a result in \cite{heins}.
\end{remark}

\subsection*{Structure}
The structure of the paper is as follows. First, in Section~\ref{S.basic}, we discuss an alternative approach to proving Hardy's result. That is, we construct a transcendental entire function whose maximum modulus set has infinitely many discontinuities. Although this does not add anything essentially new, it illustrates the ideas used in our more complicated construction.  In Section~\ref{S.construction} we lay out the technique behind the construction we shall use in our proofs. The proofs of Theorem~\ref{theo:main} and Theorem~\ref{theo:finite order} follow in the final two sections.

\subsection*{Notation}
We let $\mathbb{H}$ denote the right half-plane $\mathbb{H} \defeq \{ z \in \C \colon \Rea z > 0 \}$. For $a \in \C$ and $r > 0$, we define the ball $D(a, r) \defeq \{ z \in \C \colon |z-a| < r \}$. For a hyperbolic domain $V \subset \C$, we denote the hyperbolic distance in $V$ between two points $z, w \in V$ by $d_V(z, w)$. If $A, B, C, D \subset \C$, then we say that $A$ \emph{separates $B$ from $C$ in $D$} if $B$ and $C$ lie in different components of $D \setminus A$. We say that $A$ \emph{separates $B$ from $\infty$ in $D$} if $B$ lies in a bounded component of $D \setminus A$.

\subsection*{Acknowledgments}
We would like to thank Lasse Rempe-Gillen for many useful conversations, Lucas das Dores for valuable comments, and Vasiliki Evdoridou and the referee for helpful feedback.
%
%%%%%
%
%%%%%
%
\section{An alternative to Hardy's function}
\label{S.basic}
In this section we briefly discuss a new {\tef}, $f$, with the property that $\M(f)$ has infinitely many discontinuities. Although this example does not have any particularly novel features when compared to Hardy's example, it illustrates some of the ideas used in our more complicated construction. It also, in some sense, has the most straightforward proof to date, and so seems worth recording.

Following \cite{stallard}, set
\[
G_0 \defeq \{ x + iy \in \C \colon x > 0 \text{ and } |y| < \pi \}.
\]

For $z \in \C \setminus\overline{G_0}$, we can define a function $g$ by setting
\begin{equation}
\label{eq:f0def}
g(z) \defeq \frac{1}{2\pi i} \int_{\partial G_0} \frac{\exp(e^t)}{t-z} \ dt,
\end{equation}
where the integration is taken clockwise. It is shown in \cite{polya} that $g$ can be extended to a {\tef}, which we continue to denote by $g$, with the properties that
\begin{equation}
\label{eq:f0}
g(z) = 
\begin{cases} 
O(1/z), &\text{ for } z \notin G_0, \\
\exp(e^z) + O(1/z), &\text{ for } z \in G_0,
\end{cases}
\end{equation}
where the $O(.)$ terms above are both as $z \rightarrow \infty$.

For $n \in \N$, set $a_n \defeq 2^n$ and $b_n \defeq n(1 + 4\pi i)$, let
\begin{equation}
\label{eq:newfdef}
f(z) \defeq \sum_{n=0}^{\infty} g(a_n(z - b_n)),
\end{equation}
and define the horizontal half-strip 
\begin{equation}
\label{eq:Gndef}
G_n \defeq \{ z \in \C \colon a_n(z - b_n) \in G_0 \}.
\end{equation}

It follows from \eqref{eq:f0} that the partial sums in \eqref{eq:newfdef} converge locally uniformly, and so $f$ is a transcendental entire function. As a consequence of \eqref{eq:f0}, there exists $r_0 >0$ such that $\M(f) \setminus D(0, r_0) \subset \bigcup_{n=0}^\infty G_n$. Moreover, by \eqref{eq:newfdef}, for each $n \in \N$, we have that $$\max_{\substack{z \in G_n  \\ |z| = r}} |f(z)| < \max_{\substack{z \in G_{n+1}  \\ |z| = r}} |f(z)|$$ for all sufficiently large $r>0$. It follows that for all $r > r_0$ we have
\[
\M(f) \setminus D(0, r) \subset \bigcup_{n = p(r)}^\infty G_n,
\]
where $p(r) \rightarrow \infty$ as $r \rightarrow \infty$. The fact that $\M(f)$ has infinitely many discontinuities follows from this observation, together with the fact that the sets $G_n$ defined in \eqref{eq:Gndef} are disjoint.

Finally we remark that, with suitable modifications to the parameters in \eqref{eq:newfdef}, it is possible to gain quite a high degree of control over the moduli of the discontinuities in $\M(f)$. Roughly speaking, by reducing the parameters $a_n$ we can force discontinuities to occur close to the left-hand boundaries of the half-strips $G_n$, and by modifying the parameters $b_n$ we can ensure that the discontinuities have moduli close to prescribed values. However, it does not seem possible to get the same level of control as in the statement of Theorem~\ref{theo:main}. Moreover, the function in \eqref{eq:newfdef} is neither in class $\B$, nor of finite order. We need a more careful and general way of using Cauchy integrals to prove our two main results. 
%
%%%%%
%
%%%%%
%
\section{Background to the construction}\label{sec_prelim}
\label{S.construction}
Note that the sets $G_n$, defined in Section~\ref{S.basic}, are \emph{tracts} for the function $f$, in the sense that $|f|$ is bounded on the boundary of each $G_n$, which is an unbounded simply-connected domain; this terminology is standard when working in the class~$\B$, in which setting $f$ maps as a universal covering on the tract. The discontinuities in $\M(f)$ are achieved by forcing $\M(f)$ to ``jump'' from one tract to another. This suggests it might be possible to find a class $\B$ function with infinitely many discontinuities by using the construction in \cite{bishop}. Very roughly speaking, this result allows us to define a {\tef} in the class $\B$ by defining a collection of tracts. However, the construction of this {\tef} introduces a quasiconformal map, and it seems hard to control the effects of this map. As a result of this, we have not been able to use this construction in our setting. 

Instead, our main construction is based on results from \cite{lasse}. Unlike the result of \cite{bishop}, these only allow for a single tract, on which a function is defined. Then, proceeding as in Section 2, this function is extended to a transcendental entire map using Cauchy integrals. The careful choice of the functions involved provides  much more control of the {\tef} that results; in fact the map can be specified up to a small error function, and moreover can be chosen to lie in the class $\B$, see Theorem \ref{thm_approx1.7} below. Accordingly, we will use the result of \cite{lasse}, and ensure that the maximum modulus set of the resulting map has the desired discontinuities by constructing  a unique tract that ``winds'' forwards and backwards; see Figure~\ref{fig:V}.

We begin with some results and definitions, adapted from \cite{lasse}. 

\begin{definition}[{\cite[Definition 1.6]{lasse}}]\normalfont
A \emph{model} is a triplet $(G, V, H)$ with the following properties.
\begin{enumerate}[(i)]
\item $H \subset \C$ is a simply-connected domain that contains $\mathbb{H}$.
\item $V \subset \C$ is a simply-connected domain, disjoint from its $2\pi i$ translates.
\item $G \colon V \rightarrow H$ is a conformal isomorphism, such that if $\{z_n\}_{n\in\N}\subset H$ is a sequence with $z_n \rightarrow \infty$ in $H$, then $\Rea G^{-1}(z_n)\rightarrow +\infty$.
\end{enumerate}
\end{definition}

\begin{remark}\normalfont
The equivalent definition in \cite{lasse} defines a model function using a conformal map $\Psi \colon \exp(V) \to H$, such that $G = \Psi \circ \exp$. Our definition, and the theorem below, are equivalent, and slightly easier to work with in our setting.
\end{remark}

In this paper it is useful to make a specific choice of $H$, and so, following \cite{lasse}, we fix for the rest of the document 
\begin{equation}
\label{eq_H}
 H \defeq \{ x+iy \colon x> -14\log_+\vert y\vert \},
\end{equation}
where $\log_+(t)\defeq \max(0, \log t)$.

 Since $V$ is disjoint from its $2 \pi i$ translates, we can define a single-valued logarithm on $\exp(V)$. Using this logarithm, we can then define an analytic function $g \colon \exp(V) \to \exp(H)$ with the property that
\begin{equation}
\label{eq:gG}
g \circ \exp = \exp \circ \ G.
\end{equation} 

The following allows us to approximate the function $g$ with an entire function $f \in \B$. This result is adapted from \cite[Theorem 1.7]{lasse} together with \cite[Corollary 4.5]{lasse}.
\begin{theorem}
\label{thm_approx1.7}
Let $(G, V, H)$ be a model, with $H$ fixed in \eqref{eq_H}, and let $g$ be as in \eqref{eq:gG}. Then there exists $f\in \mathcal{B}$ such that
\begin{equation}
\label{eq_f}
f(z) \defeq 
\begin{cases}
g(z)+h(z), &\text{for } z\in \exp(V) \\
h(z),   &\text{otherwise.}
\end{cases}
\end{equation}
Here $h\colon \C \rightarrow \C$ is such that there is a constant $M > 0$, that depends only on $G^{-1}(1)$, such that $\vert h(z)\vert \leq M/\vert z \vert $, for $|z| > 1$.
\end{theorem}
\begin{remark}\normalfont
As noted earlier, the function $f$ in Theorem~\ref{thm_approx1.7} is constructed using Cauchy integrals, in a similar way to the function $g$ in \eqref{eq:f0def}. Thus it is not a coincidence that the function $h$ has similar size to the $O(.)$ terms in \eqref{eq:f0}. See \cite[Theorem 2.1]{lasse} for more details.
\end{remark}
%
%%%%%
%
%%%%%
%
\section{Proof of Theorem~\ref{theo:main}}
Suppose that $(r_n)_{n\in\N}$ is an increasing sequence of positive real numbers tending to infinity. We may assume without loss of generality that $(r_n)_{n\in\N}$ is strictly increasing, since otherwise we can choose a strictly increasing subsequence $(r_{n_k})_{n_k\in\N}$; then, if Theorem \ref{theo:main} holds for $(r_{n_k})_{n_k\in\N}$, it also holds for  $(r_n)_{n\in\N}$. We are required to construct a {\tef}, $f$, such that $\M(f)$ has a discontinuity at each of these moduli. 

For an entire function $f$, define
\[
\Mlog(f) \defeq \{ z \in \C \colon |f(e^z)| = \max_{\substack{ w \in \C \\ \Rea w = \Rea z}} |f(e^w)| \}.
\] 
It is a calculation that $\exp(\Mlog(f)) \subseteq \M(f) \subseteq \exp(\Mlog(f)) \cup \{0\}$. For each $n \in \N$, we set $x_n \defeq \log r_n$. We will prove our result by proving the equivalent fact that $\Mlog(f)$ has discontinuities at points with \emph{real part} equal to $x_n$; in other words, there exists a connected component $\Gamma$ of $\Mlog(f)$ such that $\Gamma$ contains a point with real part $x_n$, but no points of real part less than $x_n$.

The proof of this fact is complicated, and splits into three steps. First we carefully design a tract $V$, which winds backwards and forwards near each value $x_n$, and an associated Riemann map $G$; see Figure~\ref{fig:V}. The definition of $V$ depends on an infinite dimensional variable that we denote by $\uldelta$. We use $V$ and $G$, together with the results in Section~\ref{S.construction} to obtain a {\tef} $f$. Next we prove various estimates on these functions, and use these estimates to show that $\Mlog(f)$ necessarily has discontinuities close to the values $(x_n)_{n \in \N}$. Finally, we show that a suitable choice of the variable $\uldelta$ gives the required result.

We begin by defining our domain $V$. For a hyperbolic domain $U \subset \C$, we denote by $\rho_U(z)$ the hyperbolic density at the point $z \in U$. If $U$ is simply-connected, then, see \cite[Theorems 8.2 and 8.6]{beardonandminda}, we have the following standard estimate; 
\begin{equation}
\label{eq:hypest}
\frac{1}{2\operatorname{dist}(z, \partial U)} \leq \rho_U(z) \leq \frac{2}{\operatorname{dist}(z, \partial U)}, \qfor z \in U,
\end{equation}
where dist$(z, \partial U)$ is the Euclidean distance between $z$ and $\partial U$.

For $\ell > 0$, consider the rectangle
\begin{align*}
R_0 = R_0(\ell) \defeq \{ x + iy \in \C \colon -2 < x < 0 \text{ and } |y| < \ell \}.
\end{align*}

\begin{observation}
\label{claim:ell}
We can choose $\ell$ sufficiently small that the following holds. Suppose that $V \subset \C$ is an unbounded simply-connected domain, which strictly contains $R_0$, and is such that the upper, lower and left-hand boundaries of $R_0$ all lie in $\partial V$. Suppose also that $G \colon V \to H$ is a conformal map with $G(-1) = 1$, and also such that $G(z) \rightarrow \infty$ as $\Rea z \rightarrow +\infty$. Then 
\[
\max_{z \in \partial R_0 \cap V} \Rea G(z) \geq 4. 
\]
\end{observation}

\begin{proof}
This is possible because, independent of the shape of $V$, the hyperbolic distance in $V$ from $-1$ to $\partial R_0 \cap V$ tends to infinity as $\ell$ tends to zero; see \eqref{eq:hypest}. 
\end{proof}

 From now on we assume that $\ell$ has been chosen so that the implication of Observation~\ref{claim:ell} holds. Let $M > 0$ be the constant from Theorem~\ref{thm_approx1.7} with $G^{-1}(1) = -1$. Increasing $M$ if necessary, we can assume that $M > e^2$. Set 
\begin{equation}\label{eq_L}
L\defeq \log M, \quad \text{ so that} \quad L > 2.
\end{equation}

Observe that we can assume, without loss of generality, that $r_1 > e^{L+3}$ so that $x_1 > L + 3$. Otherwise, we choose $\lambda>0$ large enough that $\lambda r_1 > e^{L + 3}$. By considering the sequence $(\lambda r_n)_{n\in\N}$, we construct a function $\tilde{f}$ so that $\M(\tilde{f})$ has discontinuities at the points $\lambda r_n$. Finally, we set $f(z)\defeq \tilde{f}(\lambda z)$, and note that $\M(f)=\lambda^{-1}\M(\tilde{f})$ has discontinuities at the required moduli.

 For simplicity of notation, set $x_0 \defeq \epsilon_0 \defeq 0$. For each $n \in \N$ choose
\begin{equation}\label{eq_epsilon_n}
0 < \epsilon_n < \frac{1}{8}\min\left\{x_{n+1} - x_n, x_n - x_{n-1}, \frac{1}{2} \right\}.
\end{equation}
Let us consider the set
\[
\Delta\defeq \prod_{n\in\N} [0,\epsilon_n/8].
\]
If $\uldelta \in \Delta$, for each $n \in \N$ we shall write $\delta_n$ for the $n$th coordinate of $\uldelta$; in other words, we shall assume that $\underline{\delta} = \delta_1 \delta_2 \ldots$. 

Suppose that $\uldelta \in \Delta$. For each $n \in \N$ define six rectangles as follows:
\begin{alignat*}{3} 
R^1_n = R^1_n(\uldelta) &\defeq \{ x + iy \in \C \colon x\in [x_{n-1} + 3\epsilon_{n-1}/2, x_n + \epsilon_n/2] &&\text{ and } y\in(0,1) \},\\ 
R^2_n = R^2_n(\uldelta) &\defeq \{ x + iy \in \C \colon x\in ( x_{n} + \epsilon_{n}/2 , x_n + \epsilon_n ) &&\text{ and } y\in ( -\epsilon_n/32 , 1) \},\\
R^3_n = R^3_n(\uldelta) &\defeq \{ x + iy \in \C \colon x\in [ x_n - \delta_n + \epsilon_n/8 , x_n + \epsilon_n/2] &&\text{ and } y\in ( -\epsilon_n/32 , 0) \},\\
R^4_n = R^4_n(\uldelta) &\defeq \{ x + iy \in \C \colon x\in ( x_{n} - \delta_{n} , x_n - \delta_n + \epsilon_n/8 ) &&\text{ and } y\in ( -1 , 0) \},\\
R^5_n = R^5_n(\uldelta) &\defeq \{ x + iy \in \C \colon x\in [ x_n - \delta_n + \epsilon_n/8 , x_n + \epsilon_n] &&\text{ and } y\in ( -1 , -\epsilon_n/32) \},\\
R^6_n = R^6_n(\uldelta) &\defeq \{ x + iy \in \C \colon x\in ( x_{n} + \epsilon_{n} , x_n + 3\epsilon_n/2 ) &&\text{ and } y\in ( -1 , 1) \}.
\end{alignat*}
We then append the rectangle $R_0$ to the union of all these, to define the simply-connected domain 
\[
V = V(\underline{\delta}) \defeq R_0 \cup \bigcup_{n \in \N} \bigcup_{j=1}^6 R^j_n \setminus\{ x + iy \in \C \colon x=0 \text{ and } y\in(\ell,1) \}.
\]

\begin{figure}[htb]
\begingroup%
\makeatletter%
\providecommand\color[2][]{%
	\errmessage{(Inkscape) Color is used for the text in Inkscape, but the package 'color.sty' is not loaded}%
	\renewcommand\color[2][]{}%
}%
\providecommand\transparent[1]{%
	\errmessage{(Inkscape) Transparency is used (non-zero) for the text in Inkscape, but the package 'transparent.sty' is not loaded}%
	\renewcommand\transparent[1]{}%
}%
\providecommand\rotatebox[2]{#2}%
\newcommand*\fsize{\dimexpr\f@size pt\relax}%
\newcommand*\lineheight[1]{\fontsize{\fsize}{#1\fsize}\selectfont}%
\ifx\svgwidth\undefined%
\setlength{\unitlength}{332.18504333bp}%
\ifx\svgscale\undefined%
\relax%
\else%
\setlength{\unitlength}{\unitlength * \real{\svgscale}}%
\fi%
\else%
\setlength{\unitlength}{\svgwidth}%
\fi%
\global\let\svgwidth\undefined%
\global\let\svgscale\undefined%
\makeatother%
\begin{picture}(1,0.59999996)%
\lineheight{1}%
\setlength\tabcolsep{0pt}%
\put(0,0){\includegraphics[width=\unitlength,page=1]{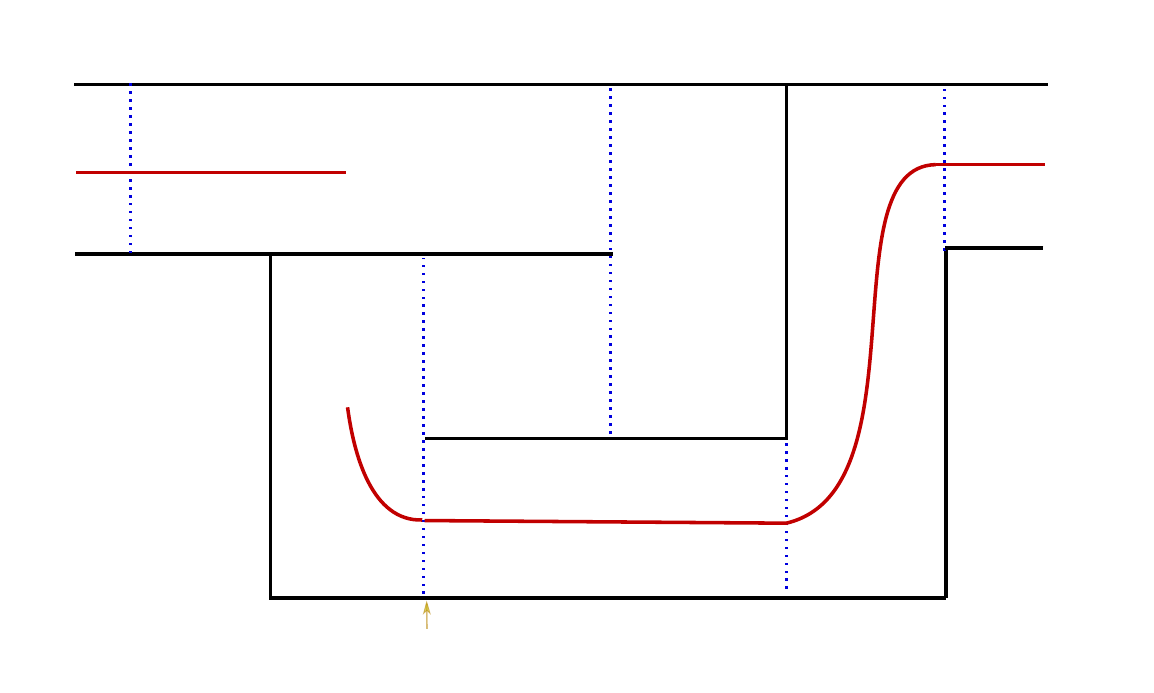}}%
\put(0.02575439,0.55905868){\color[rgb]{0,0,0}\makebox(0,0)[lt]{\lineheight{0}\smash{\begin{tabular}[t]{l}$x_{n-1}+\frac{3\epsilon_{n-1}}{2}$\end{tabular}}}}%
\put(0.17422972,0.03157987){\color[rgb]{0,0,0}\makebox(0,0)[lt]{\lineheight{0}\smash{\begin{tabular}[t]{l}$x_n-\delta_n$\end{tabular}}}}%
\put(0.47282388,0.55926888){\color[rgb]{0,0,0}\makebox(0,0)[lt]{\lineheight{0}\smash{\begin{tabular}[t]{l}$x_n+\frac{\epsilon_n}{2}$\end{tabular}}}}%
\put(0.62748059,0.55997667){\color[rgb]{0,0,0}\makebox(0,0)[lt]{\lineheight{0}\smash{\begin{tabular}[t]{l}$x_n+\epsilon_n$\end{tabular}}}}%
\put(0.76449628,0.56046583){\color[rgb]{0,0,0}\makebox(0,0)[lt]{\lineheight{0}\smash{\begin{tabular}[t]{l}$x_n+\frac{3\epsilon_n}{2}$\end{tabular}}}}%
\put(0,0){\includegraphics[width=\unitlength,page=2]{figureTractv2.pdf}}%
\put(0.0149611,0.37244393){\color[rgb]{0,0,0}\makebox(0,0)[lt]{\lineheight{0}\smash{\begin{tabular}[t]{l}$0$\end{tabular}}}}%
\put(0,0){\includegraphics[width=\unitlength,page=3]{figureTractv2.pdf}}%
\put(0.01318109,0.52056857){\color[rgb]{0,0,0}\makebox(0,0)[lt]{\lineheight{0}\smash{\begin{tabular}[t]{l}$1$\end{tabular}}}}%
\put(0,0){\includegraphics[width=\unitlength,page=4]{figureTractv2.pdf}}%
\put(0.16393651,0.07863721){\color[rgb]{0,0,0}\makebox(0,0)[lt]{\lineheight{0}\smash{\begin{tabular}[t]{l}$-1$\end{tabular}}}}%
\put(0,0){\includegraphics[width=\unitlength,page=5]{figureTractv2.pdf}}%
\put(0.09444687,0.21210398){\color[rgb]{0,0,0}\makebox(0,0)[lt]{\lineheight{0}\smash{\begin{tabular}[t]{l}$-\epsilon_n/32$\end{tabular}}}}%
\put(0.33462704,0.4579357){\color[rgb]{0,0,0.50196078}\makebox(0,0)[lt]{\lineheight{0}\smash{\begin{tabular}[t]{l}\fontsize{10pt}{1em}$R^1_n$\end{tabular}}}}%
\put(0.57811865,0.35947872){\color[rgb]{0,0,0.50196078}\makebox(0,0)[lt]{\lineheight{0}\smash{\begin{tabular}[t]{l}\fontsize{10pt}{1em}$R^2_n$\end{tabular}}}}%
\put(0.42923148,0.29762485){\color[rgb]{0,0,0.50196078}\makebox(0,0)[lt]{\lineheight{0}\smash{\begin{tabular}[t]{l}\fontsize{10pt}{1em}$R^3_n$\end{tabular}}}}%
\put(0.27704016,0.29771665){\color[rgb]{0,0,0.50196078}\makebox(0,0)[lt]{\lineheight{0}\smash{\begin{tabular}[t]{l}\fontsize{10pt}{1em}$R^4_n$\end{tabular}}}}%
\put(0.49436893,0.16563953){\color[rgb]{0,0,0.50196078}\makebox(0,0)[lt]{\lineheight{0}\smash{\begin{tabular}[t]{l}\fontsize{10pt}{1em}$R^5_n$\end{tabular}}}}%
\put(0.70323633,0.30410024){\color[rgb]{0,0,0.50196078}\makebox(0,0)[lt]{\lineheight{0}\smash{\begin{tabular}[t]{l}\fontsize{10pt}{1em}$R^6_n$\end{tabular}}}}%
\put(0.31678137,0.03064304){\color[rgb]{0,0,0}\makebox(0,0)[lt]{\lineheight{0}\smash{\begin{tabular}[t]{l}$x_n-\delta_n + \frac{\epsilon_n}{8}$\end{tabular}}}}%
\put(0,0){\includegraphics[width=\unitlength,page=6]{figureTractv2.pdf}}%
\put(0.28274203,0.56006874){\color[rgb]{0.84313725,0,0}\makebox(0,0)[lt]{\lineheight{0}\smash{\begin{tabular}[t]{l}$x_{n}$\end{tabular}}}}%
\end{picture}%
\endgroup%

\caption{A schematic of part of the simply-connected domain $V(\uldelta)$; this is not drawn to scale. The idea of our construction is to force $\Mlog(f)$, which is shown in red, to ``jump'' from $R^1_n$ to $R^4_n$ for all $n\in \N$, and we control when this ``jump'' happens by careful control of the parameter $\uldelta$. The additional rectangle $R_0$ lies to the left of this diagram.}
	\label{fig:V}
\end{figure}

See Figure \ref{fig:V}. Finally, we let $G$ be the conformal isomorphism from $V$ to $H$, such that $G(-1) = 1$, and such that $G(z) \rightarrow \infty$ as $\Rea z \rightarrow +\infty$. Note that $V$ and $G$ depend on $\uldelta$, but we suppress this dependence for simplicity. This completes the definitions of $V$ and $G$.

It follows by Observation~\ref{claim:ell} that if $\gamma \subset V \setminus R_0$ is a crosscut of $V$ that separates $-1$ from $\infty$, then
\begin{equation}
\label{eq:G}
\max_{z \in \gamma} \Rea G(z) \geq 4,
\end{equation}
 and this holds independently of $\uldelta$. (By a \emph{crosscut} of $V$ we mean a subset $A \subset V$ which is homeomorphic to the open interval $(0,1)$, such that the closure of $A$ is homeomorphic to a closed interval with exactly the two endpoints in $\partial V$.)

We next prove an estimate on $G$ that holds independently of the parameter $\uldelta$. It is helpful, for $n \in \N$, to set $$I_n \defeq (x_n - \delta_n, x_n - \delta_n + \epsilon_n/8).$$ Observe that $R^4_n = I_n \times (-1, 0)$. The significance of this interval is that we will show that $\Mlog(f)$ has a discontinuity with real part in $I_n$ whenever $\uldelta \in \Delta$.
\begin{proposition}
\label{prop_maxRe}
Suppose that $n \in \N$, and that $\uldelta \in \Delta$. Then there exists a value $t \in I_n$ such that
\begin{equation}
\label{eq:nice}
\max_{\substack{ z \in R^4_n \\ \Rea z = t'}} \operatorname{Re} G(z) > \max_{\substack{ z \in R^1_n \\ \Rea z = t'}}\operatorname{Re} G(z) + 1, \qfor t' \in [t, x_n - \delta_n + \epsilon_n/8).
\end{equation} 
\end{proposition}
\begin{proof}

Suppose that $n \in \N$ and $\uldelta \in \Delta$.
%\[
%R'_n \defeq R^3_n \cap \{ x + iy \in \C \colon x_n - \delta_n + \epsilon_n/8 < x < x_n + \epsilon_n/2 \}.
%\]
Let $\alpha_n$ be a point on the right-hand boundary of $R^3_n$ at which $\operatorname{Re} G(z)$ achieves its maximum; it is easy to see this, and other similar, maxima exist. Let $\beta_n$ be any point of the left-hand boundary of $R^3_n$. 
\begin{claim}\emph{
%\label{claim:1}
We have that} $|G(\beta_n) - G(\alpha_n)| > 1$.
\end{claim}
\begin{subproof}
Suppose, by way of contradiction, that $|G(\beta_n) - G(\alpha_n)| \leq 1$. Since $G$ is a conformal isomorphism, we have by Pick's Theorem, see \cite[Theorem~6.4]{beardonandminda}, that 
\begin{equation}
\label{eq:metrics}
d_V(\alpha_n, \beta_n) = d_H(G(\alpha_n), G(\beta_n)).
\end{equation}

We now estimate the two sides of \eqref{eq:metrics}. All points of $R^3_n$ are a Euclidean distance at most $\epsilon_n/64$ from the boundary of $V$, and so, by \eqref{eq:hypest},
\[
d_V(\alpha_n, \beta_n) \geq \frac{\operatorname{Re} \alpha_n - \operatorname{Re} \beta_n}{\epsilon_n/32} \geq \frac{\epsilon_n/8}{\epsilon_n/32} = 4.
\]

On the other hand, by \eqref{eq:G}, we know that $\operatorname{Re} G(\alpha_n) \geq 4$. Hence, by assumption, all points of the line segment from $G(\alpha_n)$ to $G(\beta_n)$ lie a distance at least $2$ from the boundary of $H$. Hence, by \eqref{eq:hypest},
\[
d_H(G(\alpha_n), G(\beta_n)) \leq \frac{2|G(\alpha_n) - G(\beta_n)|}{2} \leq 1.
\]

This is a contradiction, which completes the proof of the claim.
\end{subproof}

Next let $\alpha_n'$ be a point of the left-hand boundary of $R^3_n$ at which $\operatorname{Re} G(z)$ achieves its maximum. 
\begin{claim}\emph{We have that $\operatorname{Re} G(\alpha_n') - \operatorname{Re} G(\alpha_n) > 1$.}
\end{claim}
\begin{subproof}
This follows from the previous claim, since the image under $G$ of the left-hand boundary of $R^3_n$ separates the image under $G$ of the right-hand boundary of $R^3_n$ from infinity in $H$.
\end{subproof}

Our result follows from this claim, since $G(R^3_n)$ separates $G(R^1_n)$ from $G(R^4_n)$ in $H$, and $\partial G(R^4_n)$ contains the points of greatest real part in $\partial G(R_n^3)$.
\end{proof}

Let $g$ be the function as described in Section~\ref{sec_prelim}, corresponding to the model $(G, V, H)$. We then let $f$ and $h$ be the functions provided by Theorem \ref{thm_approx1.7}. Note that the functions $f, g$ and $h$ all depend on $\uldelta \in \Delta$, though we have suppressed this dependence for simplicity. Note that our choice of the constant $L$ in \eqref{eq_L} implies that
\begin{equation}
\label{eq:h}
|h(e^z)| < e^{-1}, \qfor \Rea z > L + 1.
\end{equation}

The following proposition essentially says that all points in $\Mlog(f)$ with sufficiently large real parts lie inside $V$, independently of the choice of $\uldelta$. Recall that  we have assumed that $x_1>L+3$.
\begin{proposition}
\label{prop_nicereal}
 We have that $\Mlog(f) \cap \{ x + iy \in \C \colon x > x_1 - 2 \epsilon_1 \} \subset V$.
\end{proposition}
\begin{proof}
Suppose that $z \notin V$, and $\Rea z > x_1 - 2 \epsilon_1 > L + 1$. Then, by \eqref{eq:h}, we have that $|f(e^z)| = |h(e^z)| < e^{-1}$. 

Next, suppose that $t > x_1 - 2 \epsilon_1$. Then $V \cap \{ t + iy \in \C \}$ contains a crosscut of $V$ that separates $-1$ from infinity, and so, by \eqref{eq:gG}, \eqref{eq:G} and \eqref{eq:h},
\[
\max_{\substack{ z \in V \\ \Rea z = t}} |f(e^z)| \geq \max_{\substack{ z \in V \\ \Rea z = t}} |g(e^z)| - \max_{\substack{ z \in V \\ \Rea z = t}} |h(e^z)| \geq \max_{\substack{ z \in V \\ \Rea z = t}} e^{\Rea G(z)} - e^{-1} \geq e^4 - e^{-1} > e^{-1}.		
\]
The result follows.
\end{proof}

We then have the following.
\begin{proposition}
\label{prop_maxmod} 
Suppose that $n \in \N$, and that $\uldelta \in \Delta$. Then there exists a value $t \in I_n$ such that
\begin{equation}
\label{eq:nicemod}
\max_{\substack{z \in R^4_n \\ \Rea z = t'}} |f(e^z)| >\max_{\substack{ z \in R^1_n \\ \Rea z = t'}} |f(e^z)|, \qfor t' \in [t, x_n - \delta_n + \epsilon_n/8).
\end{equation} 
\end{proposition}
\begin{proof}
Let $t$ be the value from Proposition~\ref{prop_maxRe}. Suppose that $t' \in [t, x_n - \delta_n + \epsilon_n/8)$. Note that, by the assumption that $x_1 > L + 3$, if $n \in \N$ and $z\in R^1_n$, then $\Rea z> L+1$. In particular, by Observation \ref{claim:ell}, $\max_{\substack{ z \in R^1_n \\ \Rea z = t'}} \exp (\Rea G(z))>4$. By this, Proposition~\ref{prop_maxRe}, \eqref{eq:gG}, \eqref{eq:G}, and \eqref{eq:h},
\begin{align*}
\max_{\substack{ z \in R^4_n \\ \Rea z = t'}} |f(e^z)| %&=  \max_{\substack{ z \in R^4_n \\ \Rea z = t}} |g(e^z) + h(e^z)| \\
 &\geq\max_{\substack{ z \in R^4_n \\ \Rea z = t'}} |g(e^z)| - e^{-1} \\
																					 &=  \max_{\substack{ z \in R^4_n \\ \Rea z = t'}} \exp \Rea G(z) - e^{-1} \\
																					 &>  \max_{\substack{ z \in R^1_n \\ \Rea z = t'}} \exp (\Rea G(z) + 1)\ - e^{-1} \\
																					 &>  \max_{\substack{ z \in R^1_n \\ \Rea z = t'}} \exp (\Rea G(z))\ + 1 \\
																					 &\geq\max_{\substack{ z \in R^1_n \\ \Rea z = t'}} |g(e^z) + h(e^z)| \\
																					 &=  \max_{\substack{ z \in R^1_n \\ \Rea z = t'}} |f(e^z)|.
\end{align*}
This completes the proof.
\end{proof}

Next, for each $n \in \N$, we define a function $\phi_n \colon \Delta \to \R$ by
\begin{equation}
\label{eq:phin}
\phi_n(\uldelta) \defeq \min \left\{t \in I_n \colon \max_{\substack{ z \in R^1_n \\ \Rea z = t}} |f(e^z)| \leq \max_{\substack{z \in R^4_n \\\Rea z = t}} |f(e^z)| \right\} - x_n. 
\end{equation}

\begin{lemma}
\label{lem_props}
The following all hold, for each $n \in \N$.
\begin{enumerate} 
\item The function $\phi_n$ is well-defined and continuous.\label{lem_props:1}
\item If $\delta_n = 0$, then $\phi_n(\uldelta) > 0$.\label{lem_props:2}
\item If $\delta_n = \epsilon_n/8$, then $\phi_n(\uldelta) < 0$.\label{lem_props:3}
\item If $\phi_n(\uldelta) = 0$, then $\Mlog(f)$ has a discontinuity with real part $x_n$.\label{lem_props:4}
\end{enumerate}
\end{lemma}
\begin{proof}
Fix $n \in \N$. First we show that $\phi_n$ is well-defined. By Proposition~\ref{prop_maxmod}, there exists $t \in I_n$ such that $$\max_{\substack{ z \in R^1_n \\ \Rea z = t}} |f(e^z)| < \max_{\substack{ z \in R^4_n \\ \Rea z = t}}|f(e^z)|, $$ 
and hence the set over which the minimum is taken in \eqref{eq:phin} is non-empty. Moreover, note that the image under $G$ of points in $V$ that are close to $\partial V$ must also be close to $\partial H$, and thus have small real part. In particular, there exists a constant $c>0$ so that if $z \in R^4_n$ and $\Rea z-(x_n - \delta_n)<c$, then $\Rea G(z)<1/3$. By this, \eqref{eq:gG} and \eqref{eq:h}, we have that if $z \in R^4_n$ and $\Rea z-(x_n - \delta_n)<c$, then
\begin{equation}
\label{eq:fnearbdry}
|f(e^z)| = |g(e^z) + h(e^z)| \leq |g(e^z)| + |h(e^z)| \leq e^{\Rea G(z)} + e^{-1} < 2.
\end{equation}
Hence, by \eqref{eq:G} and \eqref{eq:h}, for all $t\in(x_n -\delta_n,x_n -\delta_n+c)$, 
\begin{equation}\label{eq_maxIn}
\max_{\substack{ z \in R^1_n \\ \Rea z = t}} |f(e^z)| > \max_{\substack{ z \in R^1_n \\ \Rea z = t}} |g(e^z)| - e^{-1} = \exp(\max_{\substack{ z \in R^1_n \\ \Rea z = t}} \Rea G(z)) - e^{-1} > 2 > \max_{\substack{ z \in R^4_n \\ \Rea z = t}}|f(e^z)|.
\end{equation}

This means that  we can replace the non-compact interval $I_n$ in \eqref{eq:phin} by a compact interval 
\begin{equation}\label{eq_newI}
[x_n -\delta_n+c, x_n -\delta_n+c'] \subset I_n
\end{equation}
for some $c<c'<\epsilon_n/8$. It then follows by continuity of $f$ that the minimum in \eqref{eq:phin}  is attained, and so $\phi_n$ is well-defined.

Next, we observe that, in the Carath\'eodory kernel topology, the tract $V(\uldelta)$ depends continuously on $\uldelta$, using the product topology on $\Delta$. It then follows from the Carath\'eodory kernel theorem (\cite[Theorem 1.8]{Pommerenke}) that the function $G$ depends continuously on $\uldelta$ in the topology of locally uniform convergence, and in turn, so does the function $g$ from \eqref{eq:gG}. Moreover, it follows from the construction of the function $f$ in \cite{lasse}, that the function $f$ depends continuously on $\uldelta$ in the topology of locally uniform convergence; see \cite[Theorem 2.1]{lasse} and the proof of  \cite[Corollary 4.5]{lasse}. It follows that the function $\phi_n$ is continuous, and this completes the proof of \eqref{lem_props:1}.

Suppose that $\delta_n = 0$. As noted above, the minimum in \eqref{eq:phin} is attained in the interval specified in \eqref{eq_newI}, which in this case only contains points of real part greater than $x_n$. Thus $\phi_n(\uldelta) > 0$, and \eqref{lem_props:2} follows. If $\delta_n = \epsilon_n/8$, then all points in the same interval are of real part less than $x_n$. We deduce \eqref{lem_props:3}.

Finally, suppose that $\phi_n(\uldelta) = 0$. To prove \eqref{lem_props:4}, we need to show that $\Mlog(f)$ has a discontinuity at real part $x_n$. By Proposition~\ref{prop_nicereal} and the geometry of $V$, we have $\Mlog(f) \cap \{ t + iy \in \C \colon t \in I_n \} \subset R_n^1\cup R_n^4$. Moreover,
% by \eqref{eq_maxIn}, $$\Mlog(f) \cap \{ t + iy \in \C \colon t \in (x_n -\delta_n,x_n -\delta_n+c) \} \subset R_n^1.$$ In fact, by
by \eqref{eq:phin} and since $\phi_n(\uldelta) = 0$, we have that
\[
\Mlog(f) \cap \{ t + iy \in \C \colon t \in (x_n -\delta_n,x_n) \} \subset R_n^1.
\]
In addition, by the continuity of $f$, $\Mlog(f)$ contains a point $z\in R_n^4$ with real part equal to $x_n$. By Proposition~\ref{prop_maxmod}, the component $\gamma$ of $\Mlog(f)$ containing $z$ does not meet $R^1_n$. Hence, by the definition of the function $\phi_n$, $\gamma$ does not contain any points of real part smaller than $x_n$. In particular, $\Mlog(f)$ has a discontinuity at $z$, and so the result follows.
\end{proof}
As mentioned earlier, we have now shown that it is possible to construct discontinuities in $\Mlog(f)$ close to the necessary values. It remains to show that a suitable choice of $\uldelta$ leads to our result. 
\begin{lemma}
There exists $\uldelta \in \Delta$ with the property that $\phi_n(\uldelta)=0$, for all $n\in \N$. In particular, $\Mlog(f)$ has a discontinuity with real part $x_n$, for each $n\in \N$.
\end{lemma}

\begin{proof}
The proof of this lemma is closely related to \cite[Proof of Theorem 7.4]{lasse}, although we provide additional detail for the convenience of the reader. We shall, in fact, prove the stronger result that for any $A\subset \N$, there exists $\uldelta \in \Delta$ with the property that $\phi_n(\uldelta)=0$ for all $n\in A$. The result then follows by Lemma \ref{lem_props} part \eqref{lem_props:4}. It is useful to note that, by Lemma \ref{lem_props} parts \eqref{lem_props:2} and \eqref{lem_props:3}, for each $n\in \N$, there exist constants $M_n >0$ and $N_n < 0$ such that $\phi_n(\uldelta)\geq M_n$ when $\delta_n=0$, and $\phi_n(\uldelta)\leq N_n$ when $\delta_n=\epsilon_n/8$.

Clearly, if $A$ is a single point, then the claim is immediate, using the intermediate value theorem. 
	
Suppose that $A$ contains finitely many points, say $A = \{ a_1, a_2, \ldots, a_m \}$ for some $m>1$. We define a function from $\Delta(A)\defeq \prod_{n\in A} [0,\epsilon_n/8]$ to itself, as follows. First, for each $y = y_1 y_2 \ldots y_m \in \Delta(A)$, let $\uldelta(y)$ be the sequence defined by setting
\[
\delta_n \defeq
\begin{cases}
y_k, &\text{ if } n = a_k \text{ for some } k, \\
0, &\text{ otherwise}.
\end{cases}
\]
 Then, for each $n\in A$, we define a map $\xi_n \colon \Delta(A) \rightarrow [0,\epsilon_n/8]$ by
	\begin{equation*}
	\xi_n(y)\defeq 
	\begin{cases}
	0, & \text{if } \phi_n(\uldelta(y)) > M_n, \\
	{\epsilon_{n}}/{8}, & \text{if } \phi_n(\uldelta(y)) < N_n, \\
	{\epsilon_{n}}/{8} \cdot \dfrac{M_n - \phi_n(\uldelta(y))}{M_n - N_n}, & \text{otherwise.}
	\end{cases}
	\end{equation*}
	
Note that $\xi_n$ is a continuous and surjective map; this is an immediate consequence of the definition, and by Lemma \ref{lem_props} parts \eqref{lem_props:1}, \eqref{lem_props:2}, and \eqref{lem_props:3}. Hence the map
$\Phi \colon \Delta(A) \rightarrow \Delta(A)$ given by 
\[
\Phi(y)\defeq (\xi_{a_1}(y), \xi_{a_2}(y), \ldots, \xi_{a_m}(y)) 
\]
is a continuous map of an $m$-dimensional closed box to itself, that maps any face of any dimension to itself surjectively. In particular, $\Phi(\partial \Delta(A)) = \partial \Delta(A)$. 

\begin{claim}\emph{
The function $\Phi$ is a surjection.}
\end{claim}
\begin{subproof}
The proof of this result is a standard argument from algebraic topology. 

Suppose, for the sake of contradiction, that $\Phi$ omits some value $p$ in the interior of $\Delta(A)$, and let $\mathcal{H}$ be a homeomorphism from $\Delta(A)$ to the closed unit $m$-dimensional ball $B^m$, that sends $p$ to $\bm{0}=(0,\ldots,0)$. In particular, we have that $\mathcal{H}(\partial \Delta(A))=S^{m-1}\defeq \partial B^m$. Define the continuous maps $P\colon B^m\setminus \{\bm{0}\}\rightarrow S^{m-1}$ as $P(x)\defeq x/||x||$, and 
$$\mathcal{F}\defeq  P\circ \mathcal{H} \circ \Phi \circ \mathcal{H}^{-1} \colon  B^m \rightarrow S^{m-1},$$
noting that $\mathcal{F}(S^{m-1})=S^{m-1}$. Then, the composition $S^{m-1}\overset{\iota}{\hookrightarrow} B^m \xrightarrow{\mathcal{F}}S^{m-1}$ is surjective, and thus, induces a non-trivial homomorphism $\mathcal{G}$ of the homotopy group $\pi_{m-1}(S^{m-1})=\Z$ to itself. However, since $\mathcal{F} \circ \iota$ factors through $B^m$, $\mathcal{G}$ factors through the group $\pi_{m-1}(B^m)=0$, which contradicts $\mathcal{G}$ being non-trivial.
\end{subproof}

Consequently, by the surjectivity of $\Phi$, there exists a point $x\in\Delta(A)$ such that 
\[
\xi_n(x)=\epsilon_n/8 \cdot \frac{M_n}{M_n - N_n}, \qfor n\in A.
\] 
In particular, $\phi_n(\uldelta(x))=0$ for all $n\in A$, as required. 

Finally, if $A$ is infinite, then we take an increasing nested sequence of finite subsets $A_k$ that exhaust $A$, and let $\uldelta$ be a limit of the corresponding sequences. Then, the claim follows by continuity of the functions $\phi_n$ from Lemma \ref{lem_props} part~\eqref{lem_props:1}.
\end{proof}
%
%%%
%
\section{Finite order functions}
We shall use the following version of Ahlfors distortion theorem, see for example \cite[Corollary to Theorem 4.8]{ahlfors}, and compare to \cite[Theorem 3.4]{lasse}.
\begin{theorem}\label{thm_ahlfors}
	Let $V \subset \mathbb{C}$ be a simply connected domain, and
	let $\gamma_t, \gamma_{t'} \subset V$ be two maximal vertical line segments at real parts $t<t'$ respectively. Set 
	\begin{equation}\label{eq_S}
	S \defeq \{ x + iy \in \C \colon |y| < 1/2 \},
	\end{equation}
   and let $\phi \colon V \rightarrow S$ be a conformal isomorphism such that each $\phi(\gamma_t)$
	and $\phi(\gamma_{t'})$ connects the upper and lower
	boundaries of the strip $S$, and such that $\phi(\gamma_t)$ separates $\phi(\gamma_{t'})$ from $-\infty$ in $S$. For $t\leq s \leq t'$, let $\theta(s)$ denote the shortest length of a vertical line segment at real part $s$ that separates $\gamma_t$ from $\gamma_{t'}$ in $V$. If $\int_{t}^{t'}ds/\theta(s) \geq 1/2$, then 
	\begin{equation}
	\min_{z \in \gamma_{t'}} \Rea \phi(z)- \max_{w \in \gamma_{t}} \Rea \phi(w)\geq \int_{t}^{t'}\frac{ds}{\theta(s)}- \frac{1}{\pi}\ln 32.
	\end{equation}
\end{theorem}
In addition, we will make use of the following fact regarding the domain $H$ from~\eqref{eq_H}.
\begin{observation}
	\label{obs:H}
	Suppose that $z \in H$ and $x > 0$. Then $\rho_H(z) > \rho_H(z + x)$.
\end{observation}
\begin{proof}
	Let $\tau$ be the conformal map $\tau\colon H\rightarrow H, \tau(z) \defeq z + x$. Since $\tau(H) \subsetneq H$, we deduce by Pick's theorem that
	\[
	\rho_H(z) = \rho_{\tau(H)}(z + x) > \rho_H(z+x),
	\]
	as required.
\end{proof}
\begin{proof}[Proof of Theorem~\ref{theo:finite order}]
	This proof uses the same construction as Theorem~\ref{theo:main}, but with an additional restriction on the domain $V$, since Theorem~\ref{theo:finite order} has one additional hypothesis compared to Theorem~\ref{theo:main}. This is that there exists $C > 1$ such that $r_{n+1} \geq C r_n$, for $n \in \N$. This implies that the sequence $(x_n)_{n \in \N}$, defined in the proof of Theorem~\ref{theo:main} by $x_n \defeq \log r_n$, satisfies that $x_{n+1} - x_n \geq \log C > 0$, for $n \in \N$. This, in turn, implies that the terms in the sequence $(\epsilon_n)_{n\in\N}$ used in the construction in that proof, see \eqref{eq_epsilon_n}, can be chosen to be equal to a positive constant, say $\epsilon>0$. In particular, we have the following. Independent of the choice of
	\[
	\uldelta \in \Delta = \prod_{n\in\N} [0,\epsilon/8],
	\]
	there exist $\alpha > 0$, and a curve $\Gamma \subset V = V(\uldelta)$ from $-1$ to infinity, such that each point of $\Gamma$ is a Euclidean distance at least $\alpha$ from $\partial V$. Moreover, by the geometry of $V$, $\Gamma$ can be chosen so that the following holds. There is a constant $K>0$, independent of the choice of $\uldelta$, such that the Euclidean length of 
	\[
	\Gamma \cap \{ x + iy \in \C \colon x \leq t \}
	\]
	is at most $Kt$, for $t > 0$.
	
	Let $G\colon V\rightarrow H$ and $f$ be the functions resulting from the construction in the proof of Theorem~\ref{theo:main}. By Theorem~\ref{thm_approx1.7}, $f$ is in the class $\B$. The fact that $f$ is of finite order seems to be folklore. However, we have not been able to find a reference, so we prove this fact. 
	
	Let $S$ be the strip from \eqref{eq_S}, and let $\psi$ be the Riemann map from $S$ to $H$ such that $\psi(0) = 1$ and $\psi'(0) > 0$. Set $\phi = \psi^{-1} \circ G$, so that $\phi$ is a conformal map from $V$ to $S$.
	\begin{claim}\emph{
			There exists a constant $c > 0$, such that for all $t \geq 2$ we can choose a vertical crosscut $\sigma_t$ in $S$ such that
			\begin{equation}\label{eq_claim}
			\max_{\substack{ z \in V \\ \Rea z = t}} \Rea G(z) < \max_{z \in \phi^{-1}(\sigma_t)}\Rea G(z),
			\end{equation}
			and for any $\zeta_t \in  \phi^{-1}(\sigma_t)\cap \Gamma$, $d_V(-1,\zeta_t)\leq ct$.}
	\end{claim}
	
	\begin{subproof}
	  For each $t>2$, let $\gamma_t$ be the component of $V \cap \{ x + iy \in \C \colon x = t \}$ on which $\max_{z \in V}\Rea G(z)$ is attained. Note that $\gamma_t$ belongs to $R^i_n$ for some $n\geq 1$ and $i\neq 2,3$, and if $i\neq 4$, then $\gamma_t$ separates $-1$ from $\infty$ in $V$. We shall assume that $i\neq 4$, since otherwise we can take instead $\gamma_{\tilde{t}}$ for $t<\tilde{t}<t+\epsilon$, so that $\gamma_{\tilde{t}}\in R^5_n$, and the same argument applies. Note that $\gamma_{t+3}$ either belongs to $R^1_{n}$ or to $R^i_{n'}$ for some $n'\geq n+1$ and $i\geq 1$. Thus, we can choose  $t+3\leq t'\leq t+3+\epsilon$, so that $\gamma_{t'}$  separates $-1$ from $\infty$ in $V$. Then, for all $t\leq s\leq t'$, since the Euclidean length of any segment in $V \cap \{ x + iy \in \C \colon x = s \}$ is at most $2$, we have by Theorem~\ref{thm_ahlfors} that $\min_{z \in \gamma_{t'}} \Rea \phi(z)- \max_{w \in \gamma_{t}} \Rea \phi(w)>0$. Hence we can choose a vertical segment $\sigma_t$ that separates $\phi(\gamma_t)$ from $\phi(\gamma_{t'})$ in $S$. In particular, $\phi^{-1}(\sigma_t)$ separates $\gamma_t$ from infinity in $V$, and so \eqref{eq_claim} holds.%In particular, we can always choose $t'<3+\epsilon$.
	  
	  By \eqref{eq:hypest}, the hyperbolic density at any point in $\Gamma$ is at most $2/\alpha$. Moreover, since the Euclidean length of the subcurve in $\Gamma$ joining $-1$ and any $\zeta_t \in  \phi^{-1}(\sigma_t)\cap \Gamma$ is at most $Kt'$, by choosing $c\defeq 10K/\alpha$, the last part of the claim follows.
	\end{subproof}
		
		Suppose that $t \geq 2$. Let $\sigma_t$ and $\zeta_t$ be as in the previous claim, and set 
		\[
		\mu_t \defeq \max_{z \in \sigma_t} \Rea \psi(z).
		\]
		
		\begin{claim}\emph{
				If $w_t \in \sigma_t$ is such that $\Rea \psi(w_t) = \mu_t$, then $w_t$ and $\psi(w_t)$ are both real, and hence $\psi(w_t) = \mu_t$.}
		\end{claim}
		\begin{subproof}
			To see this, suppose that $\psi(w_t)$ is not real. By the symmetry of the domains $S$ and $H$, and the choice of normalisation of the Riemann map, we have that $\psi(\overline{w_t}) = \overline{\psi(w_t)}$. Note that $\overline{w_t} \in \sigma_t$. Consider the part of $\psi(\sigma_t)$ running from $\psi(w_t)$ to $\overline{\psi(w_t)}$; call this curve $\omega$. Clearly $\omega$ cannot be a line segment. Moreover, all points of $\omega$ have real part at most $\Rea \psi(w_t)$. Hence $\omega$ has greater Euclidean length than the line segment from $\psi(w_t)$ to $\overline{\psi(w_t)}$, and runs through points where the hyperbolic density is greater; see Observation~\ref{obs:H}. This is a contradiction, because $\omega$ is a geodesic. Hence $\psi(w_t)$ is real, and our choice of normalisation implies that $w_t$ is also real.
		\end{subproof}
		
		%Let $\zeta_t$ be any point of $\gamma_{t''} \cap \Gamma$. 
		%\begin{claim}\emph{
		As an application of \eqref{eq:hypest} in $H$ and using Pick's theorem, we have that
			\[
			\frac{1}{2} \log \mu_t \leq d_H(1, \mu_t) = d_S(0, w_t)\leq d_S(0,\phi(\zeta_t))= d_V(-1, \zeta_t) \leq ct, \qfor t \geq 2,
			\]
			where the middle inequality is a consequence of the simple fact that for $z \in S$, $\rho_S(z) \geq \rho_S(\Rea z)$.
			%The first inequality is an application of \eqref{eq:hypest} in $H$ and the final one was proven in the first claim of this proof. The two equalities hold by Pick's theorem, because the maps are conformal. The final inequality is an application of \eqref{eq:hypest} in $V$, using the comment earlier about the positive constant $\alpha$ and the curve $\Gamma$.
			
		%This leaves the middle inequality. It is easy to show that there is a constant $\kappa > 0$ such that $\rho_S(z) \geq \rho_S(\Rea z) = \kappa$, for $z \in S$. This implies that the geodesic in $S$ from $0$ to $w_t$ is a line segment. Let $\eta$ denote the geodesic from $0$ to $\phi(\zeta_t)$. Then, since $\Rea \phi(\zeta_t) > w_t$, we have
		%	\[
		%	d_S(0, w_t) = \kappa w_t \leq \int_{\eta} \kappa \ ds \leq \int_{\eta} \rho_S(s) \ ds = d_S(0, \phi(\zeta_t)),
		%	\]
		%	as required.
		%\end{subproof}
		By this, and using \eqref{eq_claim}, it follows that
		\[
		\max_{\substack{ z \in V \\ \Rea z = t}} \Rea G(z)  < \max_{z \in \phi^{-1}(\sigma_t)} \Rea G(z) = \mu_t \leq e^{2ct}, \qfor t \geq 2.
		\]
		It then follows by \eqref{eq:gG} and \eqref{eq_f} that the function $f$ is of finite order. 
	\end{proof}

%%%
%
%%%
%
\bibliographystyle{alpha}
\bibliography{MaxModReferences}
\end{document}